\documentclass[reqno]{amsart}
\usepackage[backend=bibtex, sorting=none, bibstyle=alphabetic, citestyle=alphabetic, sorting=nyt,maxbibnames=99]{biblatex}  
\renewbibmacro{in:}{}
\bibliography{references.bib}
\usepackage{amssymb}
\usepackage{calrsfs}
\usepackage{graphics,graphicx}
\usepackage{enumerate}
\usepackage{enumitem}
\usepackage{url}
\usepackage{xcolor}
\usepackage[ruled, lined, linesnumbered, longend]{algorithm2e}
\usepackage{hyperref}
\usepackage[justification=centering]{caption}

\newtheorem{theorem}{Theorem}[section]
\newtheorem{lemma}[theorem]{Lemma}

\newtheorem{corollary}[theorem]{Corollary}
\numberwithin{equation}{section}

\theoremstyle{remark}
\newtheorem*{remark}{Remark}

\DeclareMathOperator{\li}{li}

\def\reals{\hbox{\rm I\kern-.18em R}}
\def\complexes{\hbox{\rm C\kern-.43em
\vrule depth 0ex height 1.4ex width .05em\kern.41em}}
\def\field{\hbox{\rm I\kern-.18em F}} 

\newenvironment{section*}[2][A]{
  \section*{#2}
  \renewcommand\thesection{#1}
  \setcounter{theorem}{0}}{}

\allowdisplaybreaks

\begin{document}

\title[Estimates for the error term in the prime number theorem]{Some explicit estimates for the error term in the prime number theorem}

\author{Daniel R. Johnston and Andrew Yang}
\address{School of Science, The University of New South Wales, Canberra, Australia}
\email{daniel.johnston@adfa.edu.au}
\address{School of Science, The University of New South Wales, Canberra, Australia}
\email{andrew.yang1@adfa.edu.au}
\date\today
\keywords{}

\begin{abstract}
    By combining and improving recent techniques and results, we provide explicit estimates for the error terms $|\pi(x)-\li(x)|$, $|\theta(x)-x|$ and $|\psi(x)-x|$ appearing in the prime number theorem. For example, we show for all $x\geq 2$ that $|\psi(x)-x|\leq 9.39x(\log x)^{1.515}\exp(-0.8274\sqrt{\log x})$. Our estimates rely heavily on explicit zero-free regions and zero-density estimates for the Riemann zeta-function, and improve on existing bounds for prime-counting functions for large values of $x$.
\end{abstract}

\maketitle

\section{Introduction}
\subsection{Overview}
Let
\begin{equation*}
    \pi(x)=\sum_{p\leq x}1,\quad\theta(x)=\sum_{p\leq x}\log p\quad\text{and}\quad\psi(x)=\sum_{p^m\leq x}\log p
\end{equation*}
denote the standard prime-counting functions (see e.g.\ \cite[Chapter 4]{Apostol_13}). The prime number theorem is equivalent to any of the statements
\begin{equation*}
    \pi(x)\sim\li(x),\quad \theta(x)\sim x,\quad\text{or}\quad\psi(x)\sim x,
\end{equation*}
where $\li(x)=\int_0^x\frac{1}{\log t}\mathrm{d}t$ is the logarithmic integral. At the core of many results in analytic number theory are sharp estimates on the functions
\begin{equation}\label{errortermeqs}
    |\pi(x)-\li(x)|,\quad |\theta(x)-x|\quad \text{and}\quad |\psi(x)-x|.
\end{equation} 
As a result, many authors \cite{rosser1962approximate,F_K_2015,Trudgian_16,Axlernew2018,buthe2018analytic,Dusart_18,P_T_2021,BKLNW_21} have obtained explicit bounds on these functions. In particular, Platt and Trudgian \cite[Theorem 1]{P_T_2021} recently employed an argument due to Pintz \cite[Theorem 1]{Pintz_80} to provide estimates of the form
\begin{equation}\label{generalpsiest}
    |\psi(x)-x|\leq A(\log x)^B\exp(-C\sqrt{\log x})
\end{equation}
for some positive constants $A$, $B$ and $C$, and $x\geq\exp(1000)$. They also give related estimates for $|\theta(x)-x|$ and $|\pi(x)-\li(x)|$ \cite[Corollaries 1 and 2]{P_T_2021}.

In this paper we make refinements to Platt and Trudgian's method, in turn giving improvements to their estimates that hold for a wider range of $x$. Most notably, we modify a technique recently employed by Broadbent et al. \cite[Section A.2]{BKLNW_21} and make use of a new explicit error term for the Riemann-von Mangoldt formula \cite{cully2021error}. We also give improvements to other preliminary results in \cite{P_T_2021} by incorporating a recent verification of the Riemann hypothesis up to height $3\cdot 10^{12}$ \cite{P_T-RH_21}. Some of these improvements are also discussed in \cite[Section 6]{cully2021error}.

\subsection{Statement of results}
We obtain the following bounds for $|\psi(x)-x|$.

\begin{theorem}\label{psithm}
    For all $x\geq 2$,
    \begin{equation}
        |\psi(x)-x|\leq 9.39x(\log x)^{1.515}\exp(-0.8274\sqrt{\log x})\label{psiest1}.
    \end{equation}
    More generally, for corresponding values of  $X$, $A$, $B$, $C$ and $\epsilon_0$ in Table \ref{errortable},
    \begin{equation}\label{psigeneral}
        |\psi(x)-x|\leq Ax(\log x)^B\exp\left(-C\sqrt{\log x}\right)
    \end{equation}
    and
    \begin{equation*}
        |\psi(x)-x|\leq \epsilon_0x
    \end{equation*}
    for all $\log x\geq X$.
\end{theorem}

Notably, our values of $\epsilon_0$ in Table \ref{errortable} are an order of magnitude smaller than the corresponding values in \cite[Table 1]{P_T_2021} or \cite[Table 1]{cully2021error}.

From Theorem \ref{psithm}, we then apply standard procedures (see Section \ref{sectthetapi}) to obtain similar results for $|\theta(x)-x|$ and $|\pi(x)-\li(x)|$ (cf. \cite[Corollaries 2 and 3]{P_T_2021}).

\begin{corollary}\label{thetacor}
    For corresponding values of $X$, $A$, $B$ and $C$ in Table \ref{errortable},
    \begin{align}
        |\theta(x)-x|\leq A_1x(\log x)^{B}\exp(-C\sqrt{\log x})\quad\text{for all}\ \log x\geq X,\label{thetaest1}
    \end{align}
    where $A_1=A+0.01$.
\end{corollary}

\begin{corollary}\label{picor}
    For all $x\geq 2$,
    \begin{align}
        |\pi(x)-\li(x)|\leq 9.59 x(\log x)^{0.515}\exp(-0.8274\sqrt{\log x})\label{piest1}.
    \end{align}
\end{corollary}

We also we give estimates of the form
\begin{equation*}
    |\psi(x)-x|\leq A(\log x)^B\exp(-C\log^{3/5}x(\log\log x)^{-1/5})
\end{equation*}
(and similarly for $|\theta(x)-x|$ and $|\pi(x)-\li(x)|$) by using an explicit Vinogradov--Korobov zero-free region for $\zeta(s)$ due to Ford \cite[Theorem 5]{Ford_2002}.

\begin{theorem}\label{fordthm}
    For all $x \ge 23$, we have
    \begin{align}
        |\psi(x)-x|&\leq 0.026x(\log x)^{1.801}\exp\left(-0.1853\log^{3/5}x(\log\log x)^{-1/5}\right)\label{psiest2}\\
        |\theta(x)-x|&\leq 0.027x(\log x)^{1.801}\exp\left(-0.1853\log^{3/5}x(\log\log x)^{-1/5}\right)\label{thetaest2}\\
        |\pi(x)-x|&\leq 0.028x(\log x)^{0.801}\exp\left(-0.1853\log^{3/5}x(\log\log x)^{-1/5}\right)\label{piest2}.
    \end{align}
\end{theorem}

Certainly, the estimates \eqref{psiest2}--\eqref{piest2} are asymptotically superior to \eqref{psiest1}--\eqref{piest1}. However, the bounds in Theorem \ref{fordthm} are worse for values of $x$ most likely to appear in applications. In particular, the bounds in Theorem \ref{psithm} are better than \eqref{psiest2} for all $\exp(59) \le x \leq \exp(2.8\cdot 10^{10})$. 

In \cite{P_T_2021}, Platt and Trudgian also apply their results to an inequality studied by Ramanujan. Namely, they use their estimates for $|\theta(x)-x|$ to show that
\begin{equation}\label{ramaneq}
    \pi^2(x)<\frac{ex}{\log x}\pi\left(\frac{x}{e}\right)
\end{equation}
for all $x\geq\exp(3915)$. By substituting \eqref{thetaest1} for $X=3000$ into the equations on page 879 of \cite{P_T_2021}, one immediately obtains the following improvement.
\begin{theorem}
    Ramanujan's inequality \eqref{ramaneq} holds for all $x\geq\exp(3361)$.
\end{theorem}
Note that \eqref{ramaneq} is conjectured to hold for all $x\geq 38,358,837,683$ (see \cite{dudek2015solving}) and has also been shown to hold for all $38,358,837,683\leq x\leq\exp(103)$ \cite{johnston2021improving}.

\def\arraystretch{1.5}
\begin{table}[h]
\centering
\caption{Values of $X$, $A$, $B$, $C$ and $\epsilon_0$ for Theorem \ref{psithm} and Corollary \ref{thetacor}. Here, $\sigma$ and $K$ are parameters that appear in the proof of Theorem \ref{psithm}.}
\begin{tabular}{|c|c|c|c|c|c|c|}
\hline
$X$ & $\sigma$ & $K$ & $A$ & $B$ & $C$ & $\epsilon_0$\\
\hline
$\log 2$ & $0.985692$ & $4$ & $9.39$ & $1.515$ & $0.8274$ & $23.17$ \\
\hline
$3000$ & $0.986688$ & $4$ & $8.86$ & $1.514$ & $0.8288$ & $3.14\cdot 10^{-14}$ \\
\hline
$4000$ & $0.988164$ & $4$ & $8.15$ & $1.512$ & $0.8309$ & $3.43\cdot 10^{-17}$ \\
\hline
$5000$ & $0.989238$ & $4$ & $7.65$ & $1.511$ & $0.8324$ & $8.14\cdot 10^{-20}$ \\
\hline
$6000$ & $0.990000$ & $4$ & $7.22$ & $1.510$ & $0.8335$ & $3.35\cdot 10^{-22}$ \\
\hline
$7000$ & $0.990718$ & $4$ & $6.99$ & $1.510$ & $0.8345$ & $2.14\cdot 10^{-24}$ \\
\hline
$8000$ & $0.991258$ & $4$ & $6.78$ & $1.509$ & $0.8353$ & $1.89\cdot 10^{-26}$ \\
\hline
$9000$ & $0.991714$ & $4$ & $6.58$ & $1.509$ & $0.8359$ & $2.22\cdot 10^{-28}$ \\
\hline
$10000$ & $0.992100$ & $5$ & $6.72$ & $1.508$ & $0.8369$ & $3.27\cdot 10^{-30}$ \\
\hline
$10^5$ & $0.997312$ & 1 & $23.13$ & $1.503$ & $0.8659$ & $9.12\cdot 10^{-111}$\\
\hline
$10^6$ & $0.998974$ & 1 & $38.57$ & $1.502$ & $1.0318$ & $3.12\cdot 10^{-438}$\\
\hline
$10^7$ & $0.999662$ & 1 & $42.90$ & $1.501$ & $1.0706$ & $6.62\cdot 10^{-1459}$\\
\hline
$10^8$ & $0.999890$ & 1 & $44.41$ & $1.501$ & $1.0839$ & $2.18\cdot 10^{-4694}$\\
\hline
$10^9$ & $0.999964$ & 1 & $44.97$ & $1.501$ & $1.0886$ & $5.86\cdot 10^{-14936}$\\
\hline
$10^{10}$ & $0.999988$ & 1 & $45.17$ & $1.501$ & $1.0903$ & $3.45\cdot 10^{-47335}$\\
\hline
\end{tabular}
\label{errortable}
\end{table}

\subsection{Outline of paper}
The structure of the paper is as follows. In Section \ref{sectlem}, we state some preliminary lemmas. In Section \ref{sect3}, we prove Theorem \ref{psithm}. In Section \ref{sectthetapi} we prove Corollaries \ref{thetacor} and \ref{picor}. In Section \ref{vksect} we prove Theorem \ref{fordthm}. Finally, in Section \ref{improvesect} we detail potential improvements to our results. An appendix is also included for lengthy bounding arguments required in Sections \ref{largesect} and \ref{vkpsisect}.

\subsection{Remarks on recent correspondence and literature}\label{corressect}
Upon announcing the first version of this article, H. Kadiri informed us that she was working on similar work with A. Fiori and J. Swidinsky. Subsequently, they announced their work \cite{fiori2022density}. Compared with the results of this paper, \cite{fiori2022density} contains better bounds on $|\psi(x)-x|$ for lower values of $X$ $(X\leq 10000)$ and at present, does not feature any corresponding results for $|\pi(x)-\li(x)|$. However, as \cite{fiori2022density} is currently unpublished, and our results on $|\psi(x)-x|$ for low $X$ are used to establish subsequent results, we have refrained from using the work in \cite{fiori2022density} here. 

Importantly though, H. Kadiri and G. Hiary brought to our attention an unreliable result in \cite{hiary2016explicit} which was previously used in this work. In particular, in \cite{hiary2016explicit} an error in \cite{cheng2004explicit} was used to compute bounds on $|\zeta(1/2+it)|$. Fortunately, this error can be remedied to give the slightly worse result (cf. \cite[Theorem 1.1]{hiary2016explicit})
\begin{equation}\label{newhiary}
    \left|\zeta(1/2+it)\right|\leq 0.77t^{1/6}\log t,\quad t\geq 3.
\end{equation}
For more detail on how to obtain this modified result, see \cite[Section 2.2.1]{patel2021explicit}. The bound \eqref{newhiary} has been incorporated throughout this paper, in particular in Lemmas \ref{zerodenlem}, \ref{fordclassicregion} and \ref{fordclassicregion2} where \cite[Theorem 1.1]{hiary2016explicit} was previously employed.

\section{Useful lemmas}\label{sectlem}
In this section, we list a series of useful lemmas. Most of the following results are estimates regarding the zeros of the Riemann zeta-function which follow from existing results in the literature. Throughout, we use the notation $f(x)=O^*(g(x))$ to mean $|f(x)|\leq g(x)$ for all $x$ under consideration. Moreover, any sum will be assumed to be over the non-trivial zeros $\rho$ of the Riemann zeta-function $\zeta(s)$.

We begin by quoting a recent rigorous verification of the Riemann hypothesis up to height $3\cdot 10^{12}$.

\begin{lemma}[Riemann height {\cite{P_T-RH_21}}]\label{rheightlem}
    Let $0<\beta<1$ and $H=3\,000\,175\,332\,800$. Then, if $\zeta(\beta+i t)=0$ and $|t|\leq H$ we have $\beta=\frac{1}{2}$.
\end{lemma}

Lemma \ref{rheightlem} allows us to update some useful bounds on $|\psi(x)-x|$, $|\theta(x)-x|$ and $|\pi(x)-\li(x)|$ due to B\"uthe \cite{buthe2016estimating}.
\begin{lemma}\label{butlem1}
    The following estimates hold:
    \begin{align}
        |\psi(x)-x|&<\frac{\sqrt{x}}{8\pi}\log^2x,&\text{for }59< x\leq 2.169\cdot10^{25},\notag\\
        |\theta(x)-x|&<\frac{\sqrt{x}}{8\pi}\log^2x,&\text{for }599< x\leq 2.169\cdot10^{25},\notag\\
        |\pi(x)-\li(x)|&<\frac{\sqrt{x}}{8\pi}\log x,&\text{for }2657<x\leq 2.169\cdot10^{25}.\notag
    \end{align}
\end{lemma}
\begin{proof}
    Substitute $T=H$ (from Lemma \ref{rheightlem}) into \cite[Theorem 2]{buthe2016estimating}.
\end{proof}
\begin{lemma}\label{butlem2}
    For all $x\geq\exp(2000)$, we have
    \begin{equation*}
        \left|\frac{\psi(x)-x}{x}\right|\leq 1.570\cdot 10^{-12}.
    \end{equation*}
\end{lemma}
\begin{proof}
    Repeat the computations in \cite[Section 6]{buthe2016estimating} using the new Riemann height (Lemma \ref{rheightlem}) along with parameters\footnote{As noted in \cite[Theorem 16]{BKLNW_21} there is a small error in \cite[Theorem 1]{buthe2016estimating}. However, in the case when $\alpha=0$ there is no difference to the result.} $\alpha=0$, $c=33.6$ and $\epsilon=1.12\cdot 10^{-11}$. 
\end{proof}

Next we give a recent estimate on the error term in the Riemann-von Mangoldt formula due to Cully-Hugill and the first author \cite{cully2021error}. Here we convert their result into a specific form that is useful for our application.
\begin{lemma}\label{chdjerr}
    Let $T$ and $x$ be such that $\max\{50,\log x\}<T/1.8<(x^{1/35}-2)/4$. Then,
    \begin{equation}\label{oureq}
        \left|\frac{\psi(x)-x}{x}\right|\leq\sum_{\substack{|\Im(\rho)|\leq T}}\frac{x^{\Re(\rho)-1}}{|\Im(\rho)|}+\frac{4.3128}{T}\log^{0.6}x
    \end{equation}
    for all $x\geq\exp(1000)$.
\end{lemma}
\begin{proof}
    Using the values of $\alpha$, $\omega$, $M$ and $x_M$ given in the final row of \cite[Table 5]{cully2021error}, we have by \cite[Theorem 1.2]{cully2021error} that
    \begin{equation*}
        \frac{\psi(x)-x}{x}=\sum_{|\Im(\rho)|\leq T^*}\frac{x^{\rho-1}}{\rho}+O^*\left(\frac{4.3128}{T}\log^{0.6}x\right)
    \end{equation*}
    for some $T^*\in[T/1.8,T]$. Applying the triangle inequality and the bound $T^*\leq T$ then gives the desired result.
\end{proof}

We also require the following bound for the sum over $1/\Im(\rho)$ up to height $T$. 

\begin{lemma}[{\cite[Lemma 2.10]{saouter2015still}, \cite[Lemma 8]{brent2021mean}}]\label{reciplem}
    If $T\geq 4\pi e$, then
    \begin{equation*}
        \frac{1}{4\pi}\log^2\left(\frac{T}{2\pi}\right)-0.9321\leq\sum_{0<\Im(\rho)\leq T}\frac{1}{\Im(\rho)}\leq\frac{1}{4\pi}\log^2\left(\frac{T}{2\pi}\right).
    \end{equation*}
\end{lemma}

Now, let $N(\sigma,T)$ denote the number of zeros in the box $\sigma<\Re(s)< 1$ and $0<\Im(s)<T$. In \cite[Lemma 4.14]{K_L_N_2018}, estimates are given for $N(\sigma,T)$ for $\sigma\in[0.75,1)$ of the form
\begin{equation}\label{densityeq}
    N(\sigma,T)\leq C_1(\sigma)T^{8(1-\sigma)/3}\log^{5-2\sigma}T+C_2(\sigma)\log^2T.    
\end{equation}
where $C_1(\sigma)$ and $C_2(\sigma)$ are positive constants. As discussed in \cite[Section 6]{cully2021error}, one can improve the values of $C_1(\sigma)$ and $C_2(\sigma)$ given in \cite{K_L_N_2018} by incorporating the recent Riemann height (Lemma \ref{rheightlem}) and replacing (3.13) in \cite{K_L_N_2018} with Theorem 2 of \cite{C_T_19}. Moreover, since the results of \cite{K_L_N_2018} rely on Hiary's unreliable bound for $|\zeta(1/2+it)|$ (see Section \ref{corressect}) we use \eqref{newhiary} in place of (3.2) in \cite{K_L_N_2018}. This amounts to setting $a_1=0.77$ and $a_2=3.161$ on page 27 of \cite{K_L_N_2018}. Making these adjustments, we recalculated $C_1(\sigma)$ and $C_2(\sigma)$ for a variety of values of $\sigma$ close to $1$.

\begin{lemma}[Zero-density estimates]\label{zerodenlem}
    For corresponding values of $\sigma$, $C_1(\sigma)$ and $C_2(\sigma)$ in Table \ref{densitytable} (see Appendix \ref{zdapp}), the bound \eqref{densityeq} holds.
\end{lemma}

The following lemmas are all explicit zero-free regions for the Riemann zeta-function. As our main results hold for a large range of $x$, it is worthwhile to use of a variety of zero-free regions that are better (i.e.\ wider) at different heights.

\begin{lemma}[Classical zero-free region {\cite{M_T_2015}}]\label{classlem}
    For $|t|\geq 2$ there are no zeros of $\zeta(\beta+it)$ in the region $\beta\geq 1-\nu_1(t)$ where
    \begin{equation*}
        \nu_1(t)=\frac{1}{R_0\log |t|}
    \end{equation*}
    and\footnote{This value of $R_0$ is lower than that appearing in \cite[Theorem 1]{M_T_2015}. However, since the Riemann hypothesis has now been verified to a higher height (Lemma \ref{rheightlem}), we can take $R_0=5.5666305$ as discussed in \cite[Section 6.1]{M_T_2015}.} $R_0=5.5666305$.
\end{lemma}

\begin{lemma}\label{fordclassicregion}
    For $|t|\geq 5.45\cdot 10^8$ there are no zeros of $\zeta(\beta+it)$ in the region
    \begin{equation*}
        \beta\geq 1-\frac{1}{R(|t|)\log |t|},
    \end{equation*}
    where
    \begin{equation*}
        R(t)=\frac{J(t)+0.685+0.155\log\log t}{\log t\left(0.04962-\frac{0.0196}{J(t)+1.15}\right)},
    \end{equation*}
    with
    \begin{equation*}
        J(t)=\frac{1}{6}\log t+\log\log t+\log(0.77).
    \end{equation*}
\end{lemma}
\begin{proof}
    Same as \cite[Theorem 3]{Ford_2002} with an improved expression for $J(t)$. To achieve this, we replace Ford's (1.6) with the estimate \eqref{newhiary} given in Section \ref{corressect}.
\end{proof}
We convert this result into a slightly smaller zero-free region that is easier to work with and holds for all $|t|\geq 3$.

\begin{lemma}\label{fordclassicregion2}
For $|t| \geq 3$, there are no zeroes of $\zeta(\beta + it)$ in the region $\beta \ge 1 - \nu_2(t)$, where 
\begin{equation}\label{fordclassiclem}
    \nu_2(t) = \frac{1}{3.359\log |t|}\left(1 - \frac{8.02\log\log |t|}{\log |t|}\right).
\end{equation}
\end{lemma}

\begin{proof}
Due to the symmetry of zeros of $\zeta(s)$ about the line $\Im(s)=0$, it suffices to prove the lemma for positive $t\geq 3$. So, first we note that for $3\leq t\leq\exp(91.2853)$ the result holds by Lemmas \ref{rheightlem} and \ref{classlem}. Now, let $t \geq\exp(91.2853)$, $R(t)$ and $J(t)$ be as defined in Lemma \ref{fordclassicregion}, and 
\begin{align*}
    a_1(t)&:=\frac{0.0196}{0.04962(J(t)+1.15)}\leq\frac{6\cdot 0.0196}{0.04962\log t}\leq \frac{0.526\log\log t}{\log t},\\
    a_2(t)&:=\frac{6(1.155\log\log t+\log(0.77)+0.685)}{\log t}\leq \frac{7.494\log\log t}{\log t}.
\end{align*}
Then,
\begin{align*}
    \frac{1}{R(t)\log t}&=\frac{6\cdot 0.04962}{\log t}\left(\frac{1-a_1(t)}{1+a_2(t)}\right)\\
    &=\frac{6\cdot 0.04962}{\log t}\left(1-\frac{a_1(t)+a_2(t)}{1+a_2(t)}\right)\\
    &\geq\frac{1}{3.359\log t}\left(1-\frac{8.02\log\log t}{\log t}\right)
\end{align*}
as required.
\end{proof}
\begin{lemma}[Vinogradov--Korobov zero-free region {\cite[Theorem 5]{Ford_2002}}]\label{vklem}
    For $|t|\geq 3$ there are no zeros of $\zeta(\beta+it)$ in the region $\beta\geq 1-\nu_3(t)$ where
    \begin{equation}\label{vkregion}
        \nu_3(t)=\frac{1}{c\log^{2/3}|t|(\log\log |t|)^{1/3}}
    \end{equation}
    and $c=57.54$.
\end{lemma}

By way of comparing the zero-free regions above, let
\begin{equation*}
    \nu(t)=\max\{\nu_1(t),\nu_2(t),\nu_3(t)\}.
\end{equation*}
Then, 
\begin{align*}
    \nu(t)&=\nu_1(t)\text{ for }3\le |t| \le \exp(91.2),\\
    \nu(t)&=\nu_2(t)\text{ for }\exp(91.3)\le |t|\le \exp(54563),\\
    \nu(t)&=\nu_3(t)\text{ for }|t|\ge \exp(54563.1).\\
\end{align*}

\section{Proof of Theorem \ref{psithm}}\label{sect3}
\subsection{Estimates for small values of $x$ ($2\leq x\leq \exp(2488)$)}\label{lowsect}
For $2\leq x\leq 59$, Theorem \ref{psithm} can be verified by a straight-forward computation. For $59<x\leq\exp(50)$, Theorem \ref{psithm} holds by Lemma \ref{butlem1}. Then, to cover the range $\exp(50)<x\leq\exp(2000)$ we use \cite[Table 8]{BKLNW_21}. Finally, for $\exp(2000)<x\leq \exp(2488)$ we apply Lemma \ref{butlem2}.

\subsection{Proof of Theorem \ref{psithm} for $X\leq 10000$}\label{medsect}
Let $X\leq 10000$ correspond to one of the initial rows of Table \ref{errortable}. Suppose $x\geq x_0$ where $x_0=\exp(X)$ except in the case $X=\log 2$ whereby we set $x_0=\exp(2488)$ and use Section \ref{lowsect} to cover smaller values of $x$. In what follows, we build on the methods from \cite[Section 3]{P_T_2021} and \cite[Section A.2]{BKLNW_21}. To begin with, we set $T=\exp(2\sqrt{\log x/R_0})$ where $R_0=5.5666305$ as in Lemma \ref{classlem}. By Lemma \ref{chdjerr}
\begin{equation}\label{rvmerror}
        \left|\frac{\psi(x)-x}{x}\right|\leq\sum_{\substack{|\Im(\rho)|\leq T}}\frac{x^{\Re(\rho)-1}}{|\Im(\rho)|}+\frac{4.3128}{T}\log^{0.6}x
\end{equation}
where the sum is over all non-trivial zeros $\rho$ of the Riemann zeta-function with $|\Im(\rho)|\leq T$. For some choice of $\sigma\in[0.98,1)$, we write the sum in \eqref{rvmerror} as
\begin{equation}\label{sumspliteq}
    \sum_{|\Im(\rho)|\leq T}\frac{x^{\Re(\rho)-1}}{|\Im(\rho)|}=\sum_{\substack{|\Im(\rho)|\leq T \\\Re(\rho)\leq\sigma}}\frac{x^{\Re(\rho)-1}}{|\Im(\rho)|}+\sum_{\substack{|\Im(\rho)|\leq T \\ \Re(\rho)>\sigma}}\frac{x^{\Re(\rho)-1}}{|\Im(\rho)|}.
\end{equation}
Let $H$ be the Riemann height from Lemma \ref{rheightlem}. By Lemma \ref{reciplem}, the first sum on the right-hand side of \eqref{sumspliteq} satisfies
\begin{align}\label{s1eq}
    \sum_{\substack{|\Im(\rho)|\leq T \\ \Re(\rho)\leq\sigma}}\frac{x^{\Re(\rho)-1}}{|\Im(\rho)|}&=\sum_{\substack{|\Im(\rho)|\leq H \\ \Re(\rho)\leq\sigma}}\frac{x^{\Re(\rho)-1}}{|\Im(\rho)|}+\sum_{\substack{H<|\Im(\rho)|\leq T \\ \Re(\rho)\leq\sigma}}\frac{x^{\Re(\rho)-1}}{|\Im(\rho)|}\notag\\
    &\leq \frac{x^{-\frac{1}{2}}\log^2(\frac{H}{2\pi})}{2\pi}+x^{\sigma-1}\left(\frac{\log^2(\frac{T}{2\pi})}{2\pi}-\frac{\log^2(\frac{H}{2\pi})}{2\pi}+1.8642\right)\notag\\
    &=s_1(x,\sigma),\ \text{say}.
\end{align}
Now, let $\nu_1(t)=1/(R_0\log t)$ be as defined in Lemma \ref{classlem} and $t_0 := \exp(\sqrt{\log x / R_0})$. Note also that since $x\geq\exp(2488)$, we have $T>H$. For the second sum on the right-hand side of \eqref{sumspliteq},
\begin{align}
    \sum_{\substack{|\Im(\rho)| \leq T \\\Re(\rho) > \sigma}} \frac{x^{\Re(\rho) - 1}}{|\Im(\rho)|}&\le \sum_{\substack{|\Im(\rho)| \leq T \\ \Re(\rho) > \sigma}} \frac{x^{-\nu_1(\Im(\rho))}}{|\Im(\rho)|}\notag\\
    &= 2\int_{H}^{T}\frac{x^{-\nu_1(t)}}{t}\text{d}N(\sigma, t)\notag\\
    &\leq 2\left(\int_{H}^{t_0}\frac{x^{-\nu_1(t)}}{t}\text{d}N(\sigma, t) + \int_{t_0}^{T}\frac{x^{-\nu_1(t)}}{t}\text{d}N(\sigma, t)\right)\label{tailsumeq}
\end{align}
with the understanding that if $t_0 < H$, we remove the first integral. Next, let $K\geq 1$ and define
\begin{equation*}
    t_k:=\exp\left(\left(1+\frac{k}{K}\right)\sqrt{\frac{\log x}{R_0}}\right)    
\end{equation*}
for $0\leq k\leq K$. Noting that $x^{-\nu_1(t)}/t$ is decreasing if and only if $t\geq t_0$, the second integral in \eqref{tailsumeq} is bounded by
\begin{align}\label{ndiff}
\int_{t_0}^{T}\frac{x^{-\nu_1(t)}}{t}\text{d}N(\sigma, t) &\le \frac{x^{-\nu_1(t_0)}}{t_0}\int_{t_0}^{t_1}dN(\sigma,t)+\sum_{k = 1}^{K - 1}\frac{x^{-\nu_1(t_k)}}{t_k}\int_{t_k}^{t_{k + 1}}\text{d}N(\sigma, t)\notag\\
&\leq\frac{x^{-\nu_1(t_0)}}{t_0}(N(\sigma,t_1)-N(\sigma,t_0))+\sum_{k = 1}^{K - 1}\frac{x^{-\nu_1(t_k)}}{t_k}N(\sigma, t_{k+1})\notag\\
&=\sum_{k = 0}^{K -1}\frac{x^{-\nu_1\left(t_k\right)}}{t_k}N(\sigma, t_{k+1})-\frac{x^{-\nu_1\left(t_0\right)}}{t_0}N(\sigma, t_0).
\end{align}
Meanwhile, the first integral in \eqref{tailsumeq} is bounded by
\begin{align*}
    \int_{H}^{t_0}\frac{x^{-\nu_1(t)}}{t}\text{d}N(\sigma, t)\le \frac{x^{-\nu_1(t_0)}}{t_0}\int_H^{t_0}\text{d}N(\sigma, t) = \frac{x^{-\nu_1(t_0)}}{t_0}N(\sigma, t_0)
\end{align*}
since $N(\sigma, H) = 0$.

Therefore, writing $N_0(\sigma, t_k) = C_1(\sigma)T^{8(1-\sigma)/3}\log^{5-2\sigma}t_k+C_2(\sigma)\log^2t_k$ as the upper bound for $N(\sigma,t_k)$ from Lemma \ref{zerodenlem}, we have
\begin{align}\label{s2eq}
    \sum_{\substack{|\Im(\rho)| \leq T \\ \Re(\rho) > \sigma}}\frac{x^{\Re(\rho) - 1}}{|\Im(\rho)|} &\leq 2\sum_{k = 0}^{K -1}\frac{x^{-\nu_1\left(t_k\right)}}{t_k}N_0(\sigma, t_{k+1})\notag\\
    &=s_2(x,\sigma,K),\ \text{say}.
\end{align}
Substituting everything back into \eqref{rvmerror} gives
\begin{align}
    \left|\frac{\psi(x) - x}{x}\right| \leq s_1(x,\sigma) + s_2(x, \sigma, K) + s_3(x),\label{broadbentbound}
\end{align}
where $s_1(x,\sigma)$ and $s_2(x,\sigma,K)$ are as in \eqref{s1eq} and \eqref{s2eq}, and
\begin{equation}\label{s3eq}
    s_3(x)=\frac{4.3128}{T}\log^{0.6}x.  
\end{equation}

We now note that for fixed $K$, each term
\begin{equation*}
    2\frac{x^{-\nu_1(t_k)}}{t_k}N_0(\sigma,t_{k+1})    
\end{equation*}
appearing in $s_2(x,\sigma,K)$ decreases towards
\begin{equation}\label{asymeq}
    2C_1(\sigma)\left(1+\frac{k+1}{K}\right)^{5-2\sigma}\left(\frac{\log x}{R_0}\right)^{B}\exp\left(-C_k\sqrt{\frac{\log x}{R_0}}\right)
\end{equation}
for sufficiently large $x$, where
\begin{equation}\label{bandc}
    B=\frac{5-2\sigma}{2}\quad\text{and}\quad C_k=\frac{K+k}{K}+\frac{K}{K+k}-\frac{8}{3}(1-\sigma)\left(1+\frac{k+1}{K}\right).
\end{equation}
Therefore, if we let
\begin{equation}
    C'=\min_{k\in\{0,\ldots,K-1\}}C_k
\end{equation}
and
\begin{equation}\label{feq}
    F(x,\sigma)=\left[\left(\frac{\log x}{R}\right)^{B}\exp\left(-C'\sqrt{\frac{\log x}{R}}\right)\right]^{-1}
\end{equation}
then $s_2(x,\sigma,K)F(x,\sigma)$ is decreasing for sufficiently large $x$. Note in particular that\footnote{In fact, for all the cases we consider, $C'=C_0$.}
\begin{equation*}
    C'\leq C_0=\frac{8\sigma-2}{3}-\frac{8}{3K}(1-\sigma)<2.
\end{equation*} 
Thus, $s_1(x,\sigma)F(x,\sigma)$ and $s_3(x)F(x,\sigma)$ are also decreasing for sufficiently large $x$. We then define
\begin{equation*}
    A'=A'(x,\sigma,K):=s_1(x,\sigma)F(x,\sigma)+s_2(x,\sigma,K)F(x,\sigma)+s_3(x)F(x,\sigma)
\end{equation*}
so that provided $A'(x,\sigma,K)$ is decreasing for $x\geq x_0$, we have
\begin{equation*}
    \left|\frac{\psi(x)-x}{x}\right|\leq A'(x_0,\sigma,K)\left(\frac{\log x}{R}\right)^{B}\exp\left(-C'\sqrt{\frac{\log x}{R}}\right),
\end{equation*}
for all $x\geq x_0$. To compute the values of $A$, $B$ and $C$ in Table \ref{errortable}, we let $B$ be as in \eqref{bandc},  $A=\frac{A'}{R^{B}}$, $C=\frac{C'}{\sqrt{R}}$, and optimised over $K$ and $\sigma$ for different values of $x_0=\exp(X)$ (or $x_0=\exp(2488)$ when $X=\log 2$). 

Note that the values of $\sigma$ are optimised to a higher precision than presented in Table \ref{densitytable}. To achieve this, we note that if $\sigma_1<\sigma<\sigma_2$, then one can take $C_1(\sigma_2)$ and $C_2(\sigma_1)$ in the zero-density bound \eqref{densityeq}. This is justified as $C_1$ and $C_2$ are respectively increasing and decreasing in $\sigma$ (using the computational method we modified from \cite{K_L_N_2018}).

We also remark that Platt and Trudgian's approach \cite[Section 3]{P_T_2021} can be viewed as a less optimised version of the case $K=1$. Namely, in our notation, Platt and Trudgian have that $B=(5-2\sigma)/2$, $C'=(16\sigma-10)/3$ and $A'\to 2.0025\cdot 2^{5-2\sigma}\cdot C_1(\sigma)$ as $x\to\infty$ (see \cite[page 875]{P_T_2021}). Comparing this to \eqref{asymeq}--\eqref{feq}, we see an improvement in $C'$ and the limiting value of $A'$. In particular, for fixed $\sigma$, large $K$ and large $x$, the value of $A'$ obtained via our method will be approximately 8 times smaller than that using Platt and Trudgian's approach.

\subsection{Proof of Theorem \ref{psithm} for $X\geq 10^5$}\label{largesect}

For these values of $X$, we use the zero-free region in Lemma \ref{fordclassiclem} to obtain a better result. As this zero-free region is more complicated to work with, we use a simplified argument equivalent to the case $K = 1$ in Section \ref{medsect}. One could parametrise $K$ as in the previous section, but the improvement will only be small as we are now considering much larger $X$. 

Analogously to the previous section, we seek to maximise the function
\begin{equation*}
    \frac{x^{-\nu_2(t)}}{t}
\end{equation*}
for $t \ge H=3\,000\,175\,332\,800$. Namely, we wish to find bounds for
\begin{equation}\label{Tbounds1}
    T:=\frac{1}{\max_{t \ge H}\frac{x^{-\nu_2(t)}}{t}}=\min_{t \ge H}tx^{\nu_2(t)}.
\end{equation}
This is done in Appendix \ref{t0Tapp} (Lemma \ref{fordclassicTbound}), whereby we show for all $x\geq x_0=\exp(X)$ that
\begin{equation}
    \exp(B_2\sqrt{\log x})\leq T\leq\exp(B_3\sqrt{\log x}),
\end{equation}
where $B_2=B_2(x_0)$ and $B_3=B_3(x_0)$ are constants given in Table \ref{A0table}. Moreover, in Lemma \ref{fordclassict0bound}, we show that if $t_0$ is the point where the maximum of $x^{-\nu_2(t)}/t$ occurs, then $x^{-\nu_2(t)}/t$ is increasing for $t\in[H,t_0)$ and decreasing for $t>t_0$.

We now show how each error term from Section \ref{medsect} changes in this setting. So firstly, by \eqref{Tbounds1}, we have for all $\log x \ge X$,
\begin{align*}
\left|\sum_{\substack{|\Im(\rho)| \leq T\\\Re(\rho) \leq \sigma}}\frac{x^{\rho - 1}}{\rho}\right| &\le \frac{x^{-\frac{1}{2}}\log^2(\frac{H}{2\pi})}{2\pi}+x^{\sigma-1}\left(\frac{\log^2(\frac{T}{2\pi})}{2\pi}-\frac{\log^2(\frac{H}{2\pi})}{2\pi}+1.8642\right)\\
&\le \frac{x^{-\frac{1}{2}}\log^2(\frac{H}{2\pi})}{2\pi}+x^{\sigma-1}\left(\frac{(B_3)^2\log x}{2\pi}-\frac{\log^2(\frac{H}{2\pi})}{2\pi}+1.8642\right)\\
&= s_1'(x, \sigma),\ \text{say}.
\end{align*}
Then,
\begin{align*}
\left|\sum_{\substack{|\Im(\rho)| \leq T\\\:\Re(\rho) > \sigma}}\frac{x^{\rho - 1}}{\rho}\right|&\leq 2\left(\int_{H}^{t_0}\frac{x^{-\nu_2(t)}}{t}\text{d}N(\sigma, t) + \int_{t_0}^{T}\frac{x^{-\nu_2(t)}}{t}\text{d}N(\sigma, t)\right)\\
&\le 2\left[\frac{x^{-\nu_2(t_0)}}{t_0}N(\sigma, T)+ \left(\frac{x^{-\nu_2(t_0)}}{t_0}N(\sigma, T) - \frac{x^{-\nu_2(t_0)}}{t_0}N(\sigma,t_0)\right)\right]\\
&= \frac{2N(\sigma, T)}{T},
\end{align*}
and,
\begin{align*}
\frac{2N(\sigma,T)}{T}&\le 2\left[C_1(\sigma)T^{(5-8\sigma)/3}\log^{5 - 2\sigma} T + \frac{C_2(\sigma)\log^2 T}{T}\right],\\
&\le 2\bigg[C_1(\sigma)\exp\left(\frac{B_2(5-8\sigma)}{3}\sqrt{\log x}\right)\left(B_2\sqrt{\log x}\right)^{5 - 2\sigma}\\
&\qquad\qquad\qquad\qquad + C_2(\sigma)\exp\left(-B_2\sqrt{\log x}\right)(B_2)^2\log x\bigg]\\
&= s_2'(x, \sigma),\ \text{say}.
\end{align*}
Finally,
\begin{align*}
\frac{4.3128}{T}\log^{0.6}x &\le \frac{4.3128}{\exp(B_2\sqrt{\log x})}\log^{0.6}x = 4.3128\log^{0.6}x\exp(-B_2\sqrt{\log x})\\
&= s_3'(x, \sigma),\ \text{say}.
\end{align*}
Then, for all $x \ge x_0$,
\begin{equation}\label{finalfordclass}
    \left|\frac{\psi(x) - x}{x}\right|\le A(x_0, \sigma)\log^{\frac{5-2\sigma}{2}}(x) \exp\left(\frac{B_2(x_0)(5 - 8\sigma)}{3}\sqrt{\log x}\right)
\end{equation}
provided
\begin{align*}
    A(x, \sigma) &:= \frac{s_1(x, \sigma) + s_2(x, \sigma) + s_3(x,\sigma)}{\log^{\frac{5-2\sigma}{2}}(x) \exp\left(\frac{B_2(x_0)(5 - 8\sigma)}{3}\sqrt{\log x}\right)}
\end{align*}
is decreasing for $x \geq x_0= \exp(X)$. Optimising over $\sigma$, we used \eqref{finalfordclass} to compute the last 6 rows of Table \ref{errortable}.

\section{Proofs of Corollaries \ref{thetacor} and \ref{picor}}\label{sectthetapi}
\subsection{Bounds on $|\theta(x)-x|$}
For $2\leq x\leq 599$, the bound on $|\theta(x)-x|$ implied by the first row in Table \ref{errortable} can be verified by a straight-forward computation. For $599<x\leq\exp(58)$, the same bound follows from Lemma \ref{butlem1}. For $x>\exp(58)$ we have \cite[Corollary 5.1]{BKLNW_21} 
\begin{equation}\label{psithetaeq}
    \psi(x)-\theta(x)<a_1x^{\frac{1}{2}}+a_2x^{\frac{1}{3}},
\end{equation}
where $a_1=1+1.93378\cdot 10^{-8}$ and $a_2=1.01718$. Thus, noting that
\begin{equation*}
    |\theta(x)-x|\leq \psi(x)-\theta(x)+|\psi(x)-x|
\end{equation*}
the bounds in Corollary \ref{thetacor} for $x>\exp(58)$ follow from \eqref{psithetaeq} and the bounds for $|\psi(x)-x|$ in Theorem \ref{psithm}.

\subsection{Bounds on $|\pi(x)-\li(x)|$}\label{pisect}
For $2\leq x\leq 2657$, the desired bound \eqref{piest1} can be verified by a straight-forward computation. For $2657<x\leq\exp(58)$, we have that \eqref{piest1} holds by Lemma \ref{butlem1}. For $x>\exp(58)$, we follow a similar procedure to Platt and Trudgian \cite[Section 4]{P_T_2021} and Dusart \cite[Section 1.7]{dusart1998}. By partial summation and integration by parts, 
\begin{align*}
\left|\pi(x) - \li(x)\right|\le \left|\frac{\theta(x) - x}{\log x}\right| + \frac{2}{\log 2} + \int_2^x\frac{\left|\theta(t) - t\right|}{t\log^2 t}\text{d}t.
\end{align*}
We write
\begin{equation*}
    \int_2^x\frac{|\theta(t) - t|}{t\log^2t}\text{d}t=I_1+I_2+I_3,
\end{equation*}
where
\begin{align*}
    I_1=\int_2^{599}\frac{|\theta(t) - t|}{t\log^2t}\text{d}t,\quad
    I_2=\int_{599}^{\exp(58)}\frac{|\theta(t) - t|}{t\log^2t}\text{d}t,\quad
    I_3=\int_{\exp(58)}^x\frac{|\theta(t) - t|}{t\log^2t}\text{d}t.
\end{align*}
By direct computation, we see that
\begin{equation*}
    I_1\leq 5.43.
\end{equation*}
Then, by Lemma \ref{butlem1}, we have
\begin{equation*}
    I_2\leq\int_{599}^{\exp(58)}\frac{1}{8\pi \sqrt{t}}\mathrm{d}t\leq 7.87\cdot 10^{12}.
\end{equation*}
Now, let $A_1 = 9.40$, $B = 1.515$ and $C = 0.8274$ be the parameters corresponding to $X=\log 2$ in Corollary \ref{thetacor}. To bound $I_3$ we first define a function 
\[
h(t) = A_1t\log^{-\alpha} t\exp\left(-Cu(t)\right),\qquad 
\]
where $u(t) = \sqrt{\log t}$ and $\alpha$ is a parameter to be optimised later. We have
\begin{align*}
    h'(t) = \frac{A_1\exp(-Cu(t))}{\log^\alpha t}\left(\frac{\log t - \alpha}{\log t} - Ctu'(t)\right)\ge A_1\log^{B - 2} t \exp(-Cu(t))
\end{align*}
provided 
\begin{equation}\label{classichreq}
\log t - \alpha - Ct\log t u'(t) \ge \log^{B + \alpha - 1} t.
\end{equation}
In particular, \eqref{classichreq} is satisfied when $\alpha = 0.45$ and $t \ge \exp(58)$. Hence
\begin{align*}
    \int_{\exp(58)}^x\frac{|\theta(t) - t|}{t\log^2 t}\text{d}t &\le \int_{\exp(58)}^{x}\frac{A_1 \log^B t\exp(-Cu(t))}{\log^2 t}\text{d}t\\
    &\le \int_{\exp(58)}^{x}h'(t)\text{d}t \\
    &\leq A_1x\log^{-\alpha}x\exp(-Cu(x))\\
    &= \log^{1 - B - \alpha}x \cdot A_1x\log^{B-1} x\exp(-Cu(x))\\
    &\le \log^{1 - B - \alpha}x_0 \cdot A_1x\log^{B-1} x\exp(-Cu(x)),
\end{align*}
where $x_0=\exp(58)$ since $\log^{1 - B - \alpha}x$ is decreasing. Putting everything together, we have for all $x \ge x_0 = \exp(58)$,
\begin{equation*}
    \left|\pi(x) - \li(x)\right| \le A_2x\log^{B-1} x\exp\left(-C\sqrt{\log x}\right),
\end{equation*}
where $A_2$ is given by
\begin{align*}
&A_1\left(1 + \log^{1 - B - \alpha}x_0 + \left(\frac{2}{\log 2} + 5.43 + 7.87\cdot 10^{12}\right)\frac{\log^{1 - B}x_0}{A_1x_0}\exp\left(C\sqrt{\log x_0}\right)\right)\\
&\le 9.59,
\end{align*}
with $A_1 = 9.40$, $B = 1.515$, $C = 0.8274$ and $\alpha = 0.45$. This completes the proof of Corollary \ref{picor}. 

One could also produce a table of bounds for $|\pi(x)-\li(x)|$ similar to those in Theorem \ref{psithm} and Corollary \ref{thetacor}. However, since it is much more common in applications to use bounds on $|\psi(x)-x|$ or $|\theta(x)-x|$ rather than $|\pi(x)-\li(x)|$, we have refrained from performing such calculations here.

\section{Proof of Theorem \ref{fordthm}}\label{vksect}
\subsection{Proof of \eqref{psiest2}}\label{vkpsisect}
The following argument is essentially identical to that in Section \ref{largesect} except we use the Vinogradov--Korobov zero-free region in Lemma \ref{vklem} rather than the zero-free region in Lemma \ref{fordclassiclem}. We thus argue tersely as to simply highlight the outcome of using a different zero-free region.

Now, it suffices to show \eqref{psiest2} holds for $23 \le x \le \exp(59)$ and $x \ge \exp(2.8\cdot 10^{10})$ as the bounds in Theorem \ref{psithm} are sharper than \eqref{psiest2} for $\exp(59)\le x \le \exp(2.8\cdot 10^{10})$. By direct computation, \eqref{psiest2} holds for $23 \le x < 59$. For $59 \le x < \exp(59)$, the result follows from Lemma \ref{butlem1} and \cite[Table 8]{BKLNW_21}. Hence, from here onwards we assume $x\geq x_0=\exp(2.8\cdot 10^{10})$.

Let $t_0$, $T$, $B_2$ and $B_3$ be as defined in Lemmas \ref{t0_bound} and \ref{Tbounds} in Appendix \ref{t0Tapp}. Moreover, let 
\begin{equation*}
    r=r(x)=\frac{\log^{3/5}x}{(\log\log x)^{1/5}}.
\end{equation*}
Similar to before,
\begin{align*}
\left|\sum_{\substack{|\Im(\rho)| \leq T\\\Re(\rho) \leq \sigma}}\frac{x^{\rho - 1}}{\rho}\right| &\le \frac{x^{-\frac{1}{2}}\log^2(\frac{H}{2\pi})}{2\pi}+x^{\sigma-1}\left(\frac{\log^2(\frac{T}{2\pi})}{2\pi}-\frac{\log^2(\frac{H}{2\pi})}{2\pi}+1.8642\right)\\
&\le \frac{x^{-\frac{1}{2}}\log^2(\frac{H}{2\pi})}{2\pi}+x^{\sigma-1}\left(\frac{(B_3)^2r^2}{2\pi}-\frac{\log^2(\frac{H}{2\pi})}{2\pi}+1.8642\right)\\
&= s_1''(x, \sigma),\ \text{say}.
\end{align*}
Meanwhile, 
\begin{align*}
\left|\sum_{\substack{|\Im(\rho)| \leq T\\\:\Re(\rho) > \sigma}}\frac{x^{\rho - 1}}{\rho}\right|&\leq\frac{2N(\sigma, T)}{T}\\
&\le 2\bigg[C_1(\sigma)\exp\left(\frac{B_2(5-8\sigma)}{3}r\right)\left(B_2r\right)^{5 - 2\sigma}\\
&\qquad\qquad\qquad\qquad + C_2(\sigma)\exp\left(-B_2r\right)(B_2)^2r^2\bigg]\\
&= s_2''(x, \sigma),\ \text{say}.
\end{align*}
Finally, 
\begin{equation*}
\frac{4.3128}{T}\log^{0.6}x \le \frac{4.3128}{B_2r}\log^{0.6}x = s_3''(x, \sigma),\ \text{say}.
\end{equation*}
Therefore, for all $x\ge x_0=\exp(2.8\cdot 10^{10})$, 
\begin{align*}
\left|\frac{\psi(x)-x}{x}\right| &\leq A(x_0,\sigma)\cdot \left(B_2r\right)^{5 - 2\sigma}\exp\left(\frac{B_2(5 - 8\sigma)}{3}r\right)\\
&= A(x_0,\sigma)\cdot \left(B_2\frac{\log^{3/5}x}{(\log\log x)^{1/5}}\right)^{5 - 2\sigma}\exp\left(\frac{B_2(5 - 8\sigma)}{3}\frac{\log^{3/5}x}{(\log\log x)^{1/5}}\right).
\end{align*}
so long as 
\begin{align*}
    A(x,\sigma) := \left[s_1''(x, \sigma) + s_2''(x, \sigma) + s_3''(x, \sigma)\right]\left[\exp\left(\frac{B_2(5 - 8\sigma)}{3}r\right)(B_2r)^{5 - 2\sigma}\right]^{-1}
\end{align*}
is decreasing for $x \ge x_0$. Letting $\sigma = 0.9999932$ we obtain \eqref{psiest2}. We also use 
\begin{equation*}
    \frac{1}{(\log\log x)^{(5-2\sigma)/5}}\leq\frac{1}{(\log\log x_0)^{(5-2\sigma)/5}}
\end{equation*}
so that the final result is of the desired form. A factor of $(\log\log x)^{-(5-2\sigma)/5}$ could be included in the estimates \eqref{psiest2}--\eqref{piest2} (with a different leading constant). However, the authors refrained from doing so to prevent the bounds from being overly messy.

\subsection{Proof of \eqref{thetaest2}}
Same as the proof of Corollary \ref{thetaest1}, using \ref{psiest2} in place of \ref{psiest1}.

\subsection{Proof of \eqref{piest2}}
We proceed similarly to the proof of Corollary \ref{picor}. Namely, we use a simple computation and Lemma \ref{butlem1} to prove \eqref{piest2} up to $\exp(58)$. For $x>\exp(58)$ we bound $I_1$, $I_2$ and $I_3$ as defined in Section \ref{pisect}. As before, $I_1\leq 5.43$ and $I_2\leq 7.87\cdot 10^{12}$. To bound $I_3$ we set $u(t) = \log^{3/5}t(\log\log t)^{-1/5}$. Now,
\begin{align*}
    t u'(t) &= \frac{3\log \log t - 1}{5\log ^{5/2}t(\log\log t)^{6/5}} \le 1.63\cdot 10^{-5} \qquad \text{ for all } t \ge \exp(58)
\end{align*}
Thus, for $\alpha = 0.19$, $B = 1.801$, $C = 0.1853$ and $t\ge \exp(58)$
\begin{align*}
    \log t - Ct\log t u'(t) &\ge (1 - 4\cdot 10^{-6})\log t\ge \log^{0.991}t + 0.19 = \log^{B + \alpha - 1} t + \alpha.
\end{align*}
Repeating the argument in Section \ref{pisect}, we then have for $x \ge x_0 = \exp(58)$,
\begin{equation}\label{fordpibound}
\left|\pi(x) - \li(x)\right| \le A_2x\log^{B-1} x\exp\left(-C\frac{\log^{3/5}x}{(\log\log x)^{1/5}}\right),
\end{equation}
where 
\begin{align*}
A_2 &= A_1\left(1 + \log^{1 - B - \alpha}x_0 + \left(\frac{2}{\log 2} + 5.43 + 7.87\cdot 10^{12}\right)\frac{(u(x_0))^C\log^{1 - B}x_0}{A_1x_0}\right)\\
&\le 0.028,
\end{align*}
when $A_1 = 0.027, B = 1.801, C = 0.1853$ and $\alpha = 0.19$. This gives \eqref{piest2} as desired.

\section{Further possible improvements}\label{improvesect}
There are many estimates from recent literature that go into our results, several of which are frequently updated as new techniques and computational power become available. We bring particular attention to the zero-free regions in Lemmas \ref{classlem}--\ref{vklem}. For suppose one uses a classical zero-free region of the form
\begin{equation*}
    \beta\geq 1-\frac{1}{R\log |t|}
\end{equation*}
such as in Lemma \ref{classlem}. Then applying our method (or any variation on the work of Pintz \cite{Pintz_80}) one obtains an estimate
\begin{equation*}
    |\psi(x)-x|=O\left(x\exp(-C_1\sqrt{\log x})\right),
\end{equation*}
where $C_1$ is any real number less than
\begin{equation*}
    \frac{2}{\sqrt{R}}.
\end{equation*}
Similarly, using a Vinogradov--Korobov zero-free region
\begin{equation*}
    \beta\geq 1-\frac{1}{c\log^{2/3}|t|(\log\log|t|)^{1/3}}
\end{equation*}
such as in Lemma \ref{vklem}, one can prove the estimate
\begin{equation*}
    |\psi(x)-x|=O\left(x\exp(-C_2\log^{3/5}x(\log\log x)^{-1/5})\right),
\end{equation*}
where $C_2$ is any real number less than
\begin{equation*}
    \left(\frac{5}{3c^3}\right)^{1/5}\left[\left(\frac{3}{2}\right)^{2/5}+\left(\frac{2}{3}\right)^{3/5}\right].
\end{equation*}
Thus, in either case, improving on current zero-free regions will have a large impact on our results. In this direction, we note that the zero-free region due to Mossinghoff and Trudgian \cite{M_T_2015} has yet to be updated with the most recent Riemann height (Lemma \ref{rheightlem}). Moreover, Mossinghoff and Trudgian's zero-free region is stated for all $|t|\geq 2$. It would thus be very useful if their result could be improved/generalised for the larger values of $t$ required in our application.

The other zero-free regions we employ are due to Ford \cite{Ford_2002}, which have not been updated (since 2002) with recent estimates involving $\zeta(s)$.

We also remark that the main inefficiency in our method is bounding (see \eqref{ndiff})
\begin{equation}\label{ndiff2}
    \int_{t^k}^{t_{k+1}}dN(\sigma,t)=N(\sigma,t_{k+1})-N(\sigma,t_k).
\end{equation}
Namely, we use the trivial bound $N(\sigma,t_{k+1})-N(\sigma,t_k)\leq N(\sigma,t_{k+1})$. It would thus be interesting to further the work in \cite{K_L_N_2018} to produce zero-density estimates in different intervals. That is, one could bound the number of zeros of $\zeta(s)$ in the box $\sigma<\Re(s)<1$ and $T_0<\Im(s)<T$ for some choice of $T_0$ (not necessarily 0). This would lead to an improved upper bound for \eqref{ndiff2}.

\section{Acknowledgements}
Thanks to our supervisor Timothy Trudgian for his ongoing support and advice, and to Michaela Cully-Hugill for her assistance in computing the zero-density estimates in Table \ref{densitytable}. We also thank H. Kadiri and G. Hiary for their comments regarding the initial version of this article, as discussed in Section \ref{corressect}.

\newpage

\appendix 
\section{Bounds on $t_0$ and $T$}\label{t0Tapp}
In this appendix we prove results on the functions $t_0=t_0(x)$ and $T=T(x)$ as required in Sections \ref{largesect} and \ref{vkpsisect}. In what follows $H=3\,000\,175\,332\,800$ as in Lemma \ref{rheightlem}.

\begin{lemma}\label{fordclassict0bound}
Let $\nu_2(t)$ be as defined in Lemma \ref{fordclassiclem}, and $t_0$ be the value of $t$ such that $x^{-\nu_2(t)}/t$ is maximised for $t\geq H$. Then, $x^{-\nu_2(t)}/t$ is increasing for all $t\in[H,t_0)$ and decreasing for all $t>t_0$. Moreover, for all $x \ge x_0\geq\exp(10^5)$, 
\begin{equation}\label{A0A1}
    \exp(B_0\sqrt{\log x}) \le t_0 \le \exp(B_1 \sqrt{\log x})
\end{equation}
where $B_1 = (3.359)^{-1/2}$ and values of $B_0=B_0(x_0)$ are given in Table \ref{A0table}.
\end{lemma}
\begin{proof}
Let $D = 8.02$ and $R_1 = 3.359$. We have
\begin{align*}
    \frac{\text{d}}{\text{d}t}\left(\frac{x^{-\nu_2(t)}}{t}\right) = \frac{x^{-\nu_2(t)}\log x}{R_1 t^2}\left(\frac{D(1-2\log\log t)+\log t}{\log^3 t}-\frac{R_1}{\log x}\right).
\end{align*}
Thus, if $t_0$ is such that
\begin{equation*}
    \frac{D(1-2\log\log t_0)+\log t_0}{\log^3 t_0}=\frac{R_1}{\log x}
\end{equation*}
then $x^{-\nu_2(t)}/t$ is increasing for $t\in[H,t_0)$ decreasing for $t>t_0$ as required. 

We now prove \eqref{A0A1}. So, let $x\geq x_0$. Then for all $t \ge \exp(B_1\sqrt{\log x})$,
\begin{align*}
\log^3 t &\ge B_1^2\log x\log t= \frac{\log x\log t}{R_1}\ge \frac{\log x}{R_1}\left(\log t + D(1 - 2\log \log t)\right).
\end{align*}
Hence $\frac{\text{d}}{\text{d}t}\left(\frac{x^{-\eta(t)}}{t}\right) \le 0$ for $t \ge \exp(B_1\sqrt{\log x})$. That is, $t_0\leq\exp(B_1\sqrt{\log x})$. 

Now suppose $t=\exp(B_0\sqrt{\log x})$ where $B_0=B_0(x_0)$ is as in Table \ref{A0table}. Then,
\begin{equation*}
   D\frac{2\log\log t-1}{\log t}\leq D\frac{2\log(B_0\sqrt{\log x_0})-1}{B_0\sqrt{\log x_0}}=C_{x_0},\ \text{say}. 
\end{equation*}
As a result,
\begin{equation*}
    \frac{D(1-2\log\log t)+\log t}{\log^3 t}\geq\frac{(1-C_{x_0})}{\log^2 t}\geq\frac{R_1}{\log x}
\end{equation*}
provided
\begin{equation}\label{a0upper}
    B_0^2\leq\frac{1-C_{x_0}}{R_1},
\end{equation}
which is true for each $B_0$ in Table \ref{A0table}.
Therefore, $\frac{d}{dt}\left(\frac{x^{-\nu(t)}}{t}\right)\geq 0$ so that $t_0\geq t=\exp(B_0\sqrt{\log x})$ as required.
\end{proof}

\begin{lemma}\label{fordclassicTbound}
Let $\nu_2(t)$ be as defined in Lemma \ref{fordclassiclem} and
\begin{equation*}
    T=\frac{1}{\max_{t>H}\frac{x^{-\nu_2(t)}}{t}}=\min_{t>H}tx^{\nu_2(t)}.
\end{equation*} 
Then, for all $x\ge x_0\geq\exp(10^5)$,
\begin{equation*}
    \exp(B_2\sqrt{\log x}) \le T \le \exp(B_3\sqrt{\log x}),
\end{equation*} 
where $B_2 = B_2(x_0)$ and $B_3 = B_3(x_0)$ are given in Table \ref{A0table}.
\end{lemma}
\begin{proof}
By Lemma \ref{fordclassict0bound}, the maximum of $x^{-\nu(t)}/t$ occurs at $t_0 = \exp(\theta\sqrt{\log x})$ for some $\theta \in [B_0, B_1]$, with $B_1=(3.359)^{-1/2}$ and $B_0=B_0(x_0)$ as in Table \ref{A0table}. Thus,
\begin{align*}
    T &= t_0 x^{\nu_2(t_0)} \\
    &= \exp(\theta\sqrt{\log x}) \cdot \exp\left(\frac{\log x}{3.359\theta \sqrt{\log x}}\left(1 - \frac{8.02\log (\theta \sqrt{\log x})}{\theta \sqrt{\log x}}\right)\right)\\
    &= \exp\left(\sqrt{\log x}\left(\theta + \frac{1}{3.359\theta}\left(1 - \frac{8.02\log (\theta \sqrt{\log x})}{\theta \sqrt{\log x}}\right)\right)\right).
\end{align*}
Let
\begin{equation*}
    \alpha = \frac{1}{3.359}\left(1 - \frac{8.02\log(B_0 \sqrt{\log x_0})}{B_0 \sqrt{\log x_0}}\right).
\end{equation*} 
Then, 
\begin{align*}
    T \ge \exp\left(\sqrt{\log x}\left(\theta + \frac{\alpha}{\theta}\right)\right) \ge \exp\left(2\sqrt{\alpha} \cdot \sqrt{\log x}\right)
\end{align*}
since $f(\theta)=\theta+\frac{\alpha}{\theta}$ attains a minimum at $\theta=\sqrt{\alpha}$. Therefore, setting $B_2 := 2\sqrt{\alpha}$, we have $T\geq\exp(B_2\sqrt{\log x})$. Meanwhile, 
\[
    T \le \exp\left(\sqrt{\log x}\left(B_1 + \frac{1}{3.359 B_0(x_0)}\right)\right)
\]
so that setting $B_3(x_0) := B_1 + \frac{1}{3.359 B_0(x_0)}$ gives $T\leq\exp(B_3\sqrt{\log x})$ as required.
\end{proof}

\def\arraystretch{1.5}
\begin{table}[h]
\centering
\caption{Values of $X$, $B_0$, $B_2$, $B_3$ rounded to 7 decimal places, which appear in Lemma \ref{fordclassict0bound} and \ref{fordclassicTbound}.}
\begin{tabular}{|c|c|c|c|}
\hline
$\log x_0$ & $B_0$ & $B_2$ & $B_3$\\
\hline
$10^5$ & $0.3253505$ & $0.8721857$ & $1.4606625$ \\
\hline
$10^6$ & $0.4923764$ & $1.0346912$ & $1.1502603$ \\
\hline
$10^7$ & $0.5271511$ & $1.0716004$ & $1.1103741$ \\
\hline
$10^8$ & $0.5390163$ & $1.0842539$ & $1.0979426$ \\
\hline
$10^9$ & $0.5432643$ & $1.0887652$ & $1.0936237$ \\
\hline
$10^{10}$ & $0.5447895$ & $1.0903755$ & $1.0920896$ \\
\hline
\end{tabular}
\label{A0table}
\end{table}

We now prove analogous results for the zero-free region in Lemma \ref{vklem}.

\begin{lemma}
\label{t0_bound}
Let $x\geq x_0=\exp(2.8\cdot 10^{10})$, $\nu_3(t)$ be as defined in Lemma \ref{fordclassiclem}, and $t_0$ be the value of $t$ such that $x^{-\nu_3(t)}/t$ is maximised for $t\geq H$. Then, $x^{-\nu_3(t)}/t$ is increasing for all $t\in[H,t_0)$ and decreasing for all $t>t_0$. Moreover, 
\begin{equation}\label{B0B1}
    B_0\frac{\log^{3/5}x}{(\log\log x)^{1/5}} \le \log t_0 \le B_1 \frac{\log^{3/5}x}{(\log\log x)^{1/5}},
\end{equation}
where
\begin{equation}\label{b0andb1}
    B_0 = \left(\frac{2}{3c}\right)^{3/5}\left(\frac{5}{3}\right)^{1/5}= 0.07633\ldots\quad\text{and}\quad B_1 = 0.08228.
\end{equation}
\end{lemma}
\begin{proof}
Let $c = 57.54$ as in Lemma \ref{vklem}. We have 
\begin{equation*}
    \frac{\text{d}}{\text{d}t}\left(\frac{x^{-\nu_3(t)}}{t}\right)=\frac{x^{-\nu_3(t)}\log x}{3ct^2}\left(\frac{2\log\log t+1}{\log^{5/3}t(\log\log t)^{4/3}}-\frac{3c}{\log x}\right).
\end{equation*}
Thus, if $t_0$ is such that
\[
    \frac{2\log\log t_0+1}{\log^{5/3}t_0(\log\log t_0)^{4/3}}=\frac{3c}{\log x}
\]
then $x^{-\nu_3(t)}/t$ is increasing for $t\in[H,t_0)$ and decreasing for $t\geq t_0$ as required.

We now prove \eqref{B0B1}. So firstly, if $\log t=B_0\log^{3/5}x(\log\log x)^{-1/5}$, then
\begin{align*}
    \log x(1 + 2\log \log t) &= \frac{\log^{5/3}t(\log \log x)^{1/3}(1 + 2\log\log t)}{B_0^{5/3}}\\
    &= 3c\log^{5/3}t(\log\log t)^{4/3}\cdot \left(1 + \frac{1}{2\log\log t}\right) \cdot \left(\frac{\frac{3}{5}\log\log x}{\log\log t}\right)^{1/3}\\
    &\geq 3c\log^{5/3}t(\log\log t)^{4/3}\cdot \left(\frac{\frac{3}{5}\log\log x}{\log\log t}\right)^{1/3},
\end{align*}
and,
\[
\log\log t =\log B_0+\frac{3}{5}\log\log x - \frac{1}{5}\log\log\log x \le \frac{3}{5}\log\log x.
\]
Therefore, $\frac{\text{d}}{\text{d}t}\left(\frac{x^{-\nu_3(t)}}{t}\right) \geq 0$ so that $B_0\log^{3/5}x(\log\log x)^{-1/5}\leq\log t_0$ as desired.

On the other hand, we set 
\[
    B_1' = \left(\frac{2}{3c}\right)^{3/5}\left(\frac{1}{\beta}\right)^{1/5},
\]
where $\beta$ is a constant to be chosen later. For $\log t\geq B_1'\log^{3/5}x(\log\log x)^{-1/5}$,
\begin{align*}
    \log x(1 + 2\log\log t) &\leq 3c\log^{5/3}t(\log\log t)^{4/3}\cdot \left(1 + \frac{1}{2\log\log t}\right) \cdot \left(\frac{\beta \log\log x}{\log\log t}\right)^{1/3}.
\end{align*}
Moreover, 
\[
    1 + \frac{1}{2\log\log t} \leq 1.04425 = \gamma, \text{ say}.
\]
Setting $\beta=0.4125$ (so that $B_1'=0.08227\ldots$), we have for all $x\geq x_0$
\[
\log\log t\ge\log B_1'+\frac{3}{5}\log\log x-\frac{1}{5}\log\log\log x\geq \gamma^3\beta\log\log x.
\]
Thus, $\log x(1 + 2\log\log t) < 3c\log^{5/3}t(\log\log t)^{4/3}$. That is, $\frac{\text{d}}{\text{d}t}\left(\frac{x^{-\nu_3(t)}}{t}\right)\leq 0$ and
\begin{equation*}
    t_0\leq B_1'\frac{\log^{3/5}x}{(\log\log x)^{1/5}}\leq B_1\frac{\log^{3/5}x}{(\log\log x)^{1/5}}
\end{equation*}
as required.
\end{proof}

\begin{lemma}
\label{Tbounds}
Let $\nu_3(t)$ be as defined in Lemma \ref{fordclassiclem} and 
\[
T:=\frac{1}{\max_{t \ge H}\left(\frac{x^{-\nu_3(t)}}{t}\right)}=\min_{t \ge H}\left(x^{\nu_3(t)}t\right).
\]
Then for all $x \ge x_0 = \exp(2.8\cdot 10^{10})$, we have
\begin{equation}\label{b2b3eq}
B_2\frac{\log^{3/5}x}{(\log\log x)^{1/5}} \le \log T \le B_3\frac{\log^{3/5}x}{(\log\log x)^{1/5}},
\end{equation}
where $B_2 = 0.18525$ and $B_3 = 0.20680$.

\end{lemma}
\begin{proof}
By Lemma \ref{t0_bound}, the maximum of $x^{-\nu_3(t)}/t$ occurs at the point
\begin{equation*}
    t_0=\exp\left(\theta\frac{\log^{3/5}x}{(\log\log x)^{1/5}}\right)
\end{equation*}
for some $\theta$ satisfying $B_0<\theta< B_1$, where $B_0$ and $B_1$ are as in \eqref{b0andb1}. Now,
\begin{align*}
    \log\left(x^{\nu_3(t_0)}t_0\right)&=\frac{\log x}{c\left(\theta\frac{\log^{3/5}x}{(\log\log x)^{1/5}}\right)^{2/3}\log^{1/3}\left(\theta\frac{\log^{3/5}x}{(\log\log x)^{1/5}}\right)}+\theta\frac{\log^{3/5} x}{(\log\log x)^{1/5}}\\
    &\geq\frac{\log^{3/5}x(\log\log x)^{2/15}}{cB_1^{2/3}\left(\log B_1+\frac{3}{5}\log\log x-\frac{1}{5}\log\log\log x\right)^{1/3}}+B_0\frac{\log^{3/5} x}{(\log\log x)^{1/5}}\\
    &\geq\frac{\log^{3/5}x(\log\log x)^{2/15}}{cB_1^{2/3}\left(\frac{3}{5}\log\log x\right)^{1/3}}+B_0\frac{\log^{3/5} x}{(\log\log x)^{1/5}}\\
    &\ge B_2\frac{\log^{3/5} x}{(\log\log x)^{1/5}}
\end{align*}
and
\begin{align*}
    \log \left(x^{\nu_3(t_0)}t_0\right)&\leq\frac{\log^{3/5}x(\log\log x)^{2/15}}{cB_0^{2/3}\left(\log B_0+\frac{3}{5}\log\log x-\frac{1}{5}\log\log\log x\right)^{1/3}}+B_1\frac{\log^{3/5} x}{(\log\log x)^{1/5}}\\
    &\le \frac{\log^{3/5}x(\log\log x)^{2/15}}{cB_0^{2/3}\left(0.4666\log\log x\right)^{1/3}}+B_1\frac{\log^{3/5} x}{(\log\log x)^{1/5}}\\
    &\le B_3\frac{\log^{3/5} x}{(\log\log x)^{1/5}}
\end{align*}
as required.
\end{proof}

\newpage
\begin{section}[B]{Zero-density estimates}\label{zdapp}
\def\arraystretch{1.5}
\begin{table}[h]
\centering
\caption{Some values of $C_1(\sigma)$ and $C_2(\sigma)$ for Lemma \ref{zerodenlem}. In terms of the notation in \cite{K_L_N_2018}, we set $H_0=3\,000\,175\,332\,800$, $k=1$, $\mu=1.23623$ and optimise over parameters $d$, $\alpha$ and $\delta$.}
\begin{tabular}{|c|c|c|c|c|c|}
\hline
$\sigma$ & $d$ & $\alpha$ & $\delta$ & $C_1(\sigma)$ & $C_2(\sigma)$\\
\hline
0.980 & 0.3333 & 0.0633 & 0.3101 & 16.281 & 2.231\\
\hline
0.981 & 0.3333 & 0.0628 & 0.3101 & 16.337 & 2.223\\
\hline
0.982 & 0.3332 & 0.0624 & 0.3101 & 16.394 & 2.215\\
\hline
0.983 & 0.3332 & 0.0619 & 0.3102 & 16.450 & 2.207\\
\hline
0.984 & 0.3331 & 0.0614 & 0.3102 & 16.507 & 2.199\\
\hline
0.985 & 0.3331 & 0.0610 & 0.3102 & 16.564 & 2.191\\
\hline
0.986 & 0.3331 & 0.0605 & 0.3102 & 16.621 & 2.182\\
\hline
0.987 & 0.3330 & 0.0600 & 0.3102 & 16.678 & 2.175\\
\hline
0.988 & 0.3330 & 0.0595 & 0.3103 & 16.734 & 2.166\\
\hline
0.989 & 0.3329 & 0.0591 & 0.3103 & 16.791 & 2.159\\
\hline
0.990 & 0.3329 & 0.0586 & 0.3103 & 16.848 & 2.150\\
\hline
0.991 & 0.3329 & 0.0582 & 0.3103 & 16.905 & 2.142\\
\hline
0.992 & 0.3329 & 0.0577 & 0.3103 & 16.962 & 2.134\\
\hline
0.993 & 0.3328 & 0.0572 & 0.3103 & 17.019 & 2.126\\
\hline
0.994 & 0.3328 & 0.0568 & 0.3103 & 17.077 & 2.118\\
\hline
0.995 & 0.3327 & 0.0563 & 0.3104 & 17.134 & 2.110\\
\hline
0.996 & 0.3327 & 0.0559 & 0.3104 & 17.191 & 2.102\\
\hline
0.997 & 0.3326 & 0.0554 & 0.3104 & 17.248 & 2.094\\
\hline
0.998 & 0.3326 & 0.0550 & 0.3104 & 17.305 & 2.086\\
\hline
0.999 & 0.3326 & 0.0545 & 0.3104 & 17.362 & 2.077\\
\hline
1 & 0.3325 & 0.0539 & 0.3105 & 17.418 & 2.069\\
\hline
\end{tabular}
\label{densitytable}
\end{table}

\begin{remark}
    More precisely, the entry for $\sigma=1$ gives values for $C_1(\sigma)$ and $C_2(\sigma)$ in the limit $\sigma\to 1$. 
\end{remark}
\end{section}

\newpage

\printbibliography

\end{document}